\documentclass[11pt,reqno]{amsart}
\usepackage{fullpage}
\usepackage{amsmath, amsthm, amssymb}
\usepackage{amsfonts}
\usepackage[ansinew]{inputenc}
\usepackage[english]{babel}
\usepackage{enumitem}
\usepackage{comment}
\usepackage{graphicx}
\usepackage{hyperref} 
\usepackage{bbm} 
\usepackage{cite}
\usepackage{graphicx}
\usepackage{amscd}
\usepackage{xcolor}
\usepackage{bm}
\let\oldsqrt\sqrt
\def\sqrt{\mathpalette\DHLhksqrt}
\def\DHLhksqrt#1#2{%
\setbox0=\hbox{$#1\oldsqrt{#2\,}$}\dimen0=\ht0
\advance\dimen0-0.2\ht0
\setbox2=\hbox{\vrule height\ht0 depth -\dimen0}%
{\box0\lower0.4pt\box2}}

\allowdisplaybreaks

\newcommand{\R}{\mathbb{R}} 
\newcommand{\N}{\mathbb{N}} 

\newcommand{\dist}{\textnormal{dist}} 
\newcommand{\essinf}{\textnormal{ess\,inf}} 

\renewcommand{\phi}{\varphi}

\newcommand{\cD}{{\mathcal D}}
\newcommand{\cE}{{\mathcal E}}

\newcommand{\eps}{\varepsilon}
\renewcommand{\epsilon}{\varepsilon}

\theoremstyle{definition}
\newtheorem{defi}{Definition}[section]
\newtheorem{remark}[defi]{Remark}

\theoremstyle{plain} 
\newtheorem{thm}[defi]{Theorem}
\newtheorem{prop}[defi]{Proposition}
\newtheorem{lemma}[defi]{Lemma}

\numberwithin{equation}{section}

\title[a hopf lemma for the regional fractional laplacian]{A Hopf lemma for the regional fractional Laplacian}

\author[N. Abatangelo, M.M. Fall, and R.Y. Temgoua]{Nicola Abatangelo$^1$, Mouhamed Moustapha Fall$^2$, and Remi Yvant Temgoua$^{2,3}$}

\address{$^1$Dipartimento di Matematica, Alma Mater Studiorum Universit\`a di Bologna, P.zza di Porta S. Donato 5, 40126 Bologna, Italy.}

\email{nicola.abatangelo@unibo.it}

\address{$^2$African Institute for Mathematical Sciences in Senegal (AIMS Senegal), KM 2, Route de Joal, B.P. 1418 Mbour, S\'en\'egal.}

\email{mouhamed.m.fall@aims-senegal.org}

\address{$^3$Goethe-Universit\"{a}t Frankfurt am Main, Institut f\"{u}r Mathematik.
Robert-Mayer-Str. 10, 60325 Frankfurt, Germany.}

\email{temgoua@math.uni-frankfurt.de, remi.y.temgoua@aims-senegal.org}

\date{\today}

\begin{document}

\begin{abstract}
We provide a Hopf boundary lemma for the regional fractional Laplacian $(-\Delta)^s_{\Omega}$, with $\Omega\subset\R^N$ a bounded open set. More precisely, given $u$ a pointwise or weak super-solution of the equation $(-\Delta)^s_{\Omega}u=c(x)u$ in $\Omega$, we show that the ratio $u(x)/(\dist(x,\partial\Omega))^{2s-1}$ is strictly positive as $x$ approaches the boundary $\partial\Omega$ of $\Omega$. We also prove a strong maximum principle for distributional super-solutions. 
\end{abstract}

\maketitle

{\footnotesize
\begin{center}
\textit{Keywords:} regional fractional Laplacian, Hopf boundary lemma, pointwise super-solution, weak super-solution, distributional super-solution.
\end{center}
}

\section{Introduction and main results} 
Let $s\in(1/2,1)$ and let $\Omega\subset\R^N~(N\geq2)$ be a bounded domain with $C^{1,1}$ boundary. The regional fractional Laplacian $(-\Delta)^s_{\Omega}$ of a function $u:\Omega\to\R$ is defined as
\begin{align}\label{def-regional}
 (-\Delta)^s_{\Omega}u(x)=c_{N,s}P.V.\int_{\Omega}\frac{u(x)-u(y)}{{|x-y|}^{N+2s}}\;dy=c_{N,s}\lim\limits_{\epsilon\rightarrow0^+}\int_{\Omega\setminus B_{\epsilon}(x)}\frac{u(x)-u(y)}{{|x-y|}^{N+2s}}\;dy,
\end{align}
provided that the limit exists. We recall that ``$P.V.$'' stands for the Cauchy principal value and that the normalization constant $c_{N,s}$ is explicitly given by 
\[
c_{N,s}:=\bigg(\int_{\R^N}\frac{1-\cos(\zeta_1)}{{|\zeta|}^{N+2s}}\;d\zeta\bigg)^{-1}=s(1-s)\frac{2^{2s}\,\Gamma(\frac{N+2s}{2})}{\pi^{N/2}\,\Gamma(2-s)}.
\]
For functions $u$ belonging to $C^{2s+\epsilon}_{loc}(\Omega)\cap L^{\infty}(\Omega)$ for some $\epsilon>0$, the integral in \eqref{def-regional} is finite. 
In this way then, we say that \eqref{def-regional} is defined \textit{pointwisely} in $\Omega$.




The study of the regional fractional Laplacian has received some growing attention in recent years. However, in contrast to that of the\footnote{Sometimes it is also called \textit{restricted} fractional Laplacian.} fractional Laplacian
\begin{align}\label{def-fractional}
 (-\Delta)^su(x)=c_{N,s}P.V.\int_{\R^N}\frac{u(x)-u(y)}{{|x-y|}^{N+2s}}\;dy,
\end{align}
the theory of elliptic problems driven by the regional fractional Laplacian is less developed in spite of some known results. We are concerned here in particular with the Hopf boundary lemma, which is a powerful tool for the study of qualitative properties of solutions like, for example, their monotonicity and symmetry, also via moving plane arguments. 

In \cite{greco2016hopf}, the authors obtained a Hopf lemma for pointwise super-solutions for an elliptic equation involving the fractional Laplacian $(-\Delta)^s$ under the assumption that an interior ball condition holds. For the Hopf boundary lemma for weak super-solutions related to the fractional $p$-Laplacian, we refer to \cite{del2017hopf} and references therein. Other references on the Hopf boundary lemma for fractional Laplacian can be found in \cite{biswas2020hopf,caffarelli2010variational,chen2020hopf,fall2015overdetermined,jin2019hopf,li2019hopf}. However, to the best of our knowledge, an analogue result for the regional fractional Laplacian has not been investigated before. 
Let us mention here that while the Hopf lemma is usually used to run a moving plane method in the case of the fractional Laplacian, as recalled above, this does not seem to be the case for the regional fractional Laplacian. The moving plane method for $(-\Delta)^s_{\Omega}$ remains indeed a challenging question: the main difficulty relies on the fact that the operator depends on the domain and therefore, upon scaling the domain, the operator changes as well. We expect a symmetry breaking in the case of the regional fractional Laplacian defined on bounded domains.

Here, we investigate the validity of a suitable Hopf-type lemma for super-solutions of the equation
\begin{equation}\label{main-problem-to-study}
(-\Delta)^s_{\Omega}u=c(x)u\quad\text{in}~~\Omega.
\end{equation}
We analyse this both for the case of \textit{pointwise} and \textit{weak} super-solutions. Moreover, we also study a strong maximum principle for \textit{distributional} super-solutions to \eqref{main-problem-to-study}. So, before stating our main results, let us recall the following definitions (notations are defined in Section \ref{section:preliminaries}).
\begin{defi}\label{definition-of-pointwise-supersolution}
	We say that a function $u:\Omega\to\R$ is a pointwise super-solution of \eqref{main-problem-to-study}
	if $u\in C^{2s+\epsilon}_{loc}(\Omega)\cap L^{\infty}(\Omega)$ for some $\epsilon>0$ and
	\begin{equation*}
	(-\Delta)^s_{\Omega}u(x)\geq c(x)u(x)\qquad\text{for any }x\in\Omega.
	\end{equation*}
\end{defi}
\begin{defi}\label{definition-of-weak-supersolution}
	We say that a function $u:\Omega\to\R$ is a weak super-solution of \eqref{main-problem-to-study} if $u\in H^s(\Omega)$ and 
	\begin{equation*}
	\cE(u,\phi)\geq\int_{\Omega}cu\phi\qquad\text{for any }\phi\in C^{\infty}_c(\Omega),\ \phi\geq0\text{ in }\Omega.
	\end{equation*}
\end{defi}
\begin{defi}\label{dfinition-of-distributional-supersolution}
	We say that a function $u:\Omega\to\R$ is a distributional super-solution of \eqref{main-problem-to-study} if $u\in L^1(\Omega)$ and
	\begin{equation*}
	\int_{\Omega}u\,(-\Delta)^s_{\Omega}\phi\geq\int_{\Omega}cu\phi\qquad\text{for any }\phi\in C^{\infty}_c(\Omega),\ \phi\geq0\text{ in }\Omega.
	\end{equation*}
	In this case, we briefly write 
	\begin{equation*}
	(-\Delta)^s_{\Omega} u \geq c(x)u \qquad\text{in }\cD'(\Omega).
	\end{equation*}
\end{defi}

\begin{remark}
Sub-solutions can be defined in similar ways as in Definitions \ref{definition-of-pointwise-supersolution}, \ref{definition-of-weak-supersolution}, and \ref{dfinition-of-distributional-supersolution}. Also, in the case of Definition \ref{definition-of-weak-supersolution}, by density the test function $\phi$ can be chosen in $H^s_0(\Omega)_+$ if $c$ is somewhat well-behaved (see Lemma \ref{lem:tech} below for more details).
\end{remark}

We are going to denote by $\delta_\Omega(x)=\inf\{|x-\theta|:\theta\in\partial\Omega\}$ for $x\in\Omega$.
The main results of the paper are the following. 

\begin{thm}[Hopf lemma for pointwise super-solutions]\label{hopf-lemma} 
Let $\Omega\subset\R^N$ be an open bounded set with $C^{1,1}$ boundary and $s\in(1/2,1)$.
Let $c\in L^{\infty}(\Omega)$ and let $u:\overline\Omega\rightarrow\R$ be a lower semicontinuous super-solution (in the sense of Definition~\ref{definition-of-pointwise-supersolution}) of \eqref{main-problem-to-study}.
	\begin{enumerate}[label=(\roman*)]
		\item If $c\leq0$ in $\Omega$ and $u\geq0$ on $\partial\Omega$, then either $u$ vanishes identically in $\Omega$ or
		\begin{equation}\label{eq:2}
		\liminf_{\Omega\ni x\rightarrow z}\frac{u(x)}{\delta_{\Omega}(x)^{2s-1}}>0
		\qquad\text{for any }z\in\partial\Omega.
		\end{equation}
		\item If $u\geq0$ in $\overline{\Omega}$, then either $u$ vanishes identically in $\Omega$ or \eqref{eq:2} holds true.
	\end{enumerate}
\end{thm}

\begin{thm}[Hopf lemma for weak super-solutions]\label{hopf-lemma-for-weak-super-solutions}
Let $\Omega\subset\R^N$ be an open bounded set with $C^{1,1}$ boundary and $s\in(1/2,1)$.
	Let $c:\Omega\to\R$ be a measurable function and let $u\in H^s(\Omega)$ be a weak super-solution (in the sense of Definition~\ref{definition-of-weak-supersolution}) of \eqref{main-problem-to-study}.
	Suppose that either
	\begin{align}\label{ass1}
	c\in L^\infty(\Omega)
	\end{align}
	or
	\begin{align}\label{ass2}
	c\in L^q(\Omega),\ q>\frac{N}{2s},
	\qquad\text{and}\qquad 
	u\in L^\infty_{loc}(\Omega),
	\end{align}
	hold.
	\begin{enumerate}[label=(\roman*)]
		\item If $c\leq0$ in $\Omega$ and $u\geq0$ on $\partial\Omega$, then either $u$ vanishes identically in $\Omega$ or
		\begin{equation}\label{c}
		\text{there exists}\quad\eps_0>0 \qquad \text{such that} \qquad 
		\frac{u}{\delta_{\Omega}^{2s-1}}>\eps_0 \quad\text{in }\Omega.
		\end{equation}
		\item If $u\geq0$ in $\Omega$, then either u vanishes identically in $\Omega$ or \eqref{c} holds true.
	\end{enumerate}
\end{thm}


Let us first comment on the proof of Theorem \ref{hopf-lemma}. Starting with a strong maximum principle, we obtain the strict positivity of non-trivial super-solutions of \eqref{main-problem-to-study}: this is where the lower semicontinuity of $u$ is needed. In a next step, we construct a barrier from below for $u$ in terms of the torsion function $u_{tor}$, \textit{i.e.}, the solution to the boundary value problem 
\begin{align}\label{torsion-problem-for-regional}
\left\{\begin{aligned}
(-\Delta)^s_{\Omega}u_{tor}&=1 && \text{in }\Omega, \\
u_{tor}&=0 && \text{on }\partial\Omega.
\end{aligned}\right.
\end{align}
This function is known to satisfy, on smooth domains, the double-sided estimate
\begin{equation}\label{eq:8}
C^{-1}\delta_{\Omega}^{2s-1}\leq u_{tor}\leq C\delta_{\Omega}^{2s-1} 
\qquad\text{in }\Omega
\end{equation}
for some $C>1$, see \cite{bonforte2018sharp,chen2018dirichlet}
which are based on some estimates in \cite{bogdan2003censored,chen2002green,guan2006reflected}. Intuitively, \eqref{eq:8} gives that the boundary behaviour of super-solutions described by \eqref{eq:2} and \eqref{c} is optimal.
We notice that, in contrast to what happens for the fractional Laplacian, there are no explicit examples of torsion functions for the regional fractional Laplacian, even in the case when $\Omega$ is a ball. In \cite{duo2019comparative}, a numerical analysis is performed in the one-dimensional case $\Omega=(-1,1)$.  

We mention that the existence and uniqueness of pointwise and weak solutions to the Dirichlet problem \eqref{torsion-problem-for-regional} with general bounded right-hand side was obtained in \cite{chen2018dirichlet}. 
We notice also that the H\"older regularity up to the boundary of any weak solution of \eqref{torsion-problem-for-regional} was recently proved in \cite{fall2020regional}, while regularity up to the boundary of pointwise solution of \eqref{torsion-problem-for-regional} was obtained earlier in \cite{chen2018dirichlet}. We also mention that the boundary regularity of the ratio $u_{tor}/\delta_{\Omega}^{2s-1}$ has been established in \cite{fall2020regional} in the case when $\Omega$ is of class $C^{1,\beta}$ for some $\beta>0$. Thus, it makes sense to evaluate $u_{tor}/\delta_{\Omega}^{2s-1}$ pointwisely on $\overline{\Omega}$ .

The proof of Theorem \ref{hopf-lemma-for-weak-super-solutions} follows the same line of thought as the one of Theorem \ref{hopf-lemma}, although with some more technical difficulties due to the weak character of super-solutions involved. For example, when $c\in L^q(\Omega)$ the strong maximum principle involved in our strategy takes the following form. 

\begin{prop}[Strong maximum principle for distributional super-solutions]\label{strong-max-principle-distribution-regional}
	Let $\Omega\subset\R^N$ be a bounded open set and $u\in L^\infty_{\mathrm{loc}}(\Omega)$ 
	be a distributional super-solution (in the sense of Definition \ref{dfinition-of-distributional-supersolution}) of \eqref{main-problem-to-study} with 
	\begin{align}\label{q-hyp}
c\in L^q_{loc}(\Omega),\quad q>\frac{N}{2s}.
	\end{align}
	If $u\geq0$ in $\Omega$, then either $u$ vanishes identically in $\Omega$ or
	\begin{align*}
	\underset{K}{\essinf}\; u>0\qquad\text{for any }K\subset\subset\Omega.
	\end{align*}
\end{prop}


The paper is organized as follows. In Section \ref{section:preliminaries}, we present some notations and definitions. Section \ref{section:the case of pointwise solutions} is devoted to the proof of Theorem \ref{hopf-lemma}, whereas in Section \ref{section:the case of weak solutions} we prove Theorem \ref{hopf-lemma-for-weak-super-solutions}. Finally, in Section \ref{section:the case of distributional solutions} we prove Proposition \ref{strong-max-principle-distribution-regional}.
\bigskip

\textbf{Acknowledgements:} This work is supported by DAAD and BMBF (Germany) within project 57385104. The first and second author were supported by the Alexander von Humboldt Foundation. 
The authors would also like to thank Tobias Weth and Sven Jarohs for valuable discussions.

\section{Preliminaries}\label{section:preliminaries}
We collect in this section some notations and useful tools. 
For $s\in (0,1)$, $H^s(\Omega)$ denotes the space of functions $u\in L^2(\Omega)$ such that
\begin{equation*}
[u]^2_{H^s(\Omega)}:=\frac{c_{N,s}}{2}\int_{\Omega}\int_{\Omega}\frac{\big(u(x)-u(y)\big)^2}{{|x-y|}^{N+2s}}\; dx\;dy<\infty.
\end{equation*}
It is a Hilbert space endowed with the norm
\begin{equation*}
\|u\|_{H^s(\Omega)}:=\big(\|u\|^2_{L^2(\Omega)}+[u]^2_{H^s(\Omega)}\big)^{1/2}.
\end{equation*}
We denote by $H^s_0(\Omega)$ the completion of $C^{\infty}_c(\Omega)$ with respect to the norm $\|\cdot\|_{H^s(\Omega)}$. It is known that for $s\in(1/2,1)$, $H^s_0(\Omega)$ is a Hilbert space with the norm $\|\cdot\|_{H^s_0(\Omega)}=[\cdot]_{H^s(\Omega)}$ (which is equivalent to the usual one in $H^s(\Omega)$ thanks to a Poincar\'{e}-type inequality) and it can be characterized as follows
\begin{equation*}
H^s_0(\Omega):=\big\{u\in H^s(\Omega):u=0\ \text{on}\ \partial\Omega\big\}.
\end{equation*}
Next, we define $H^s_0(\Omega)_+$ by
\begin{equation*}
H^s_0(\Omega)_+:=\big\{u\in H^s_0(\Omega):u\geq0\ \text{in}\ \Omega\big\}.
\end{equation*}
For $u,v\in H^s_0(\Omega)$, we consider the symmetric, continuous, and coercive bilinear form
\begin{equation*}
\cE(u,v):=\frac{c_{N,s}}{2}\int_{\Omega}\int_{\Omega}\frac{\big(u(x)-u(y)\big)\big(v(x)-v(y)\big)}{{|x-y|}^{N+2s}}\;dx\;dy.
\end{equation*}
The first Dirichlet eigenvalue of $(-\Delta)^s_\Omega$ in $\Omega$ can be defined by
\begin{equation}\label{first-dirichlet-eigenvalue-of-regional}
\lambda_1(\Omega)=\min_{\substack{u\in H^s_0(\Omega)\\ u\neq0}}\frac{\cE(u,u)}{\|u\|^2_{L^2(\Omega)}}.
\end{equation}
It holds $\lambda_1(\Omega)>0$, with the corresponding eigenfunction unique and strictly positive in $\Omega$.

Given $x\in\Omega$ and $r>0$, we denote by $B_r(x)$ the open ball centred at $x$ with radius $r$. 
We denote by $u^+:=\max\{u,0\}$ and $u^-:=\max\{-u,0\}$ the positive and negative part of $u$ respectively. We also recall that, if $u\in H^s(\Omega)$, then $u^+,u^-\in H^s(\Omega)$ as well: this follows from a simple calculation, indeed $u=u^+-u^-$ and
\[
[u]_{H^s(\Omega)}^2=\cE(u,u)=\cE(u^+,u^+)-2\cE(u^+,u^-)+\cE(u^-,u^-)
\]
where
\begin{multline*}
\cE(u^+,u^-)=\frac{c_{N,s}}{2}\int_{\Omega}\int_{\Omega}\frac{\big(u^+(x)-u^+(y)\big)\big(u^-(x)-u^-(y)\big)}{{|x-y|}^{N+2s}}\;dx\;dy=\\
=-c_{N,s}\int_{\Omega}\int_{\Omega}\frac{u^+(x)\,u^-(y)}{{|x-y|}^{N+2s}}\;dx\;dy\leq 0.
\end{multline*}

\section{Proof of the Hopf lemma: the case of pointwise super-solutions}\label{section:the case of pointwise solutions}

The aim of this section is to prove Theorem \ref{hopf-lemma}. Before doing this, we need one key result: we state and prove a strong maximum principle for pointwise super-solutions of \eqref{main-problem-to-study}.

\begin{prop}[Strong maximum principle for pointwise super-solutions]\label{strong-minimum-principle}
Let $\Omega\subset\R^N$ be a bounded open set.
Let $c\in L^{\infty}(\Omega)$ and $u:\overline{\Omega}\rightarrow\R$ be a lower semicontinuous function super-solution (in the sense of Definition \ref{definition-of-pointwise-supersolution}) of \eqref{main-problem-to-study}.
	\begin{enumerate}[label=(\roman*)]
		\item If $c\leq0$ in $\Omega$ and $u\geq0$ on $\partial\Omega$, then either $u$ vanishes identically in $\Omega$, or $u>0$ in $\Omega$.
		\item If $u\geq0$ in $\overline{\Omega}$, then either $u$ vanishes identically in $\Omega$, or $u>0$ in $\Omega$.
	\end{enumerate}
\end{prop}

\begin{proof}
Before going into the proof, we start by proving that the function $u$ is nonnegative in $\overline{\Omega}$ as long as the hypotheses of assertion $(i)$ are satisfied.

Let us assume that $c\leq0$ in $\Omega$, $u\geq0$ on $\partial\Omega$, and that $u$ does not vanish identically on $\Omega$. Then we claim that 
\begin{equation}\label{claim}
u\geq0\quad\text{in}\quad\overline{\Omega}.
\end{equation}
Assume to the contrary that \eqref{claim} does not hold, that is, $u$ is negative somewhere in $\Omega$. Then, using that $\overline{\Omega}$ is compact together with the hypotheses of lower semicontinuity of $u$, a negative minimum of the function $u$ must be achieved in $\Omega$. In other words, there exists $x_0\in\Omega$ such that
\begin{equation}\label{negaitiv-minimum}
u(x_0)=\min_{x\in\Omega}u(x)<0.
\end{equation}
Combining \eqref{negaitiv-minimum} with $u\geq0$ on $\partial\Omega$, it follows that
\begin{equation*}
(-\Delta)^s_{\Omega}u(x_0)= c_{N,s}P.V.\int_{\Omega}\frac{u(x_0)-u(y)}{{|x_0-y|}^{N+2s}}\;dy<0.
\end{equation*}
But, since by assumption $c(x_0)\leq0$, we have that $c(x_0)u(x_0)\geq0$.
Therefore
\begin{align*}
0>(-\Delta)^s_\Omega u(x_0)\geq c(x_0)u(x_0)\geq 0.
\end{align*} 
which is a contradiction. Consequently, claim \eqref{claim} follows.

So we can now suppose $u\geq 0$ in $\Omega$. Suppose that $u\not\equiv0$ in $\Omega$ and let us prove that
\begin{equation}\label{eq:4}
u>0\qquad\text{in }\Omega.
\end{equation}
First of all, we recall that by the lower semicontinuity of $u$, there exist $x_1\in\Omega$ and $\epsilon_1,r>0$ such that
\begin{equation*} 
u(y)\geq\epsilon_1\qquad\text{for all}\quad y\in B_r(x_1)\subset\Omega.
\end{equation*}
If the inequality \eqref{eq:4} were not true, that is, if $u(\tilde{x})=0$ at some $\tilde{x}\in\Omega$, then it would hold
\begin{equation*}
(-\Delta)^s_{\Omega}u(\tilde{x})=c_{N,s}P.V.\int_{\Omega}\frac{-u(y)}{{|\tilde{x}-y|}^{N+2s}}\;dy
\leq c_{N,s}P.V.\int_{B_r(x_1)}\frac{-u(y)}{{|\tilde{x}-y|}^{N+2s}}\;dy<0
\end{equation*}
Therefore 
\begin{align*}
0>(-\Delta)^s_\Omega u(\tilde{x})\geq c(\tilde{x})u(\tilde{x})=0,
\end{align*}
a contradiction. Thus, the strict inequality $u>0$ in $\Omega$ must hold true. 
\end{proof}

Having the above strong maximum principle, we can now give the proof of Theorem \ref{hopf-lemma} by following some ideas in \cite{greco2016hopf}.

\begin{proof}[Proof of Theorem \ref{hopf-lemma}]
	From Proposition \ref{strong-minimum-principle} it follows that
	\begin{equation}\label{eq:5}
	u(x)>0\qquad\text{for all }x\in\Omega
	\end{equation}
 provided that $u$ does not vanish identically in $\Omega$. In other words, if $u$ does not vanish identically in $\Omega$, then for every compact subset $K\subset\Omega$ we have
\begin{equation}\label{eq:6}
\inf_{y\in K}u(y)>0.
\end{equation}
Now suppose that $u$ does not vanish identically in $\Omega$ and let us prove \eqref{eq:2}. To this end, it suffices to construct a barrier for $u$ in terms of the solution problem \eqref{torsion-problem-for-regional}. Let $u_{tor}$ denote the pointwise solution of \eqref{torsion-problem-for-regional}.

Next, for $n\in\N$, we set
\begin{equation}\label{barrier}
v_n(x)=\frac{1}{n}u_{tor}(x)\qquad\text{for }x\in\overline{\Omega}.
\end{equation}
Then, by definition and \eqref{eq:8}, by the boundedness of $\Omega$ it follows that 
\begin{equation}\label{eq:11}
v_n\rightarrow0\qquad\text{uniformly in }\overline{\Omega}\text{ as }n\to\infty.
\end{equation}
We wish now to show that there exists some $n_0\in\N$ such that
\begin{equation}\label{eq:10}
u\geq v_n\qquad\text{in }\overline{\Omega},\text{ for any }n\geq n_0.
\end{equation}
In order to prove \eqref{eq:10}, we argue by contradiction: suppose that for every $n\in\N$ the function $w_n$ defined by
\begin{equation*}
w_n:=v_n-u \qquad \text{in }\overline{\Omega}
\end{equation*}
is positive somewhere in $\overline{\Omega}$. Then, using that $w_n=v_n-u=-u\leq0$ on $\partial\Omega$ and the compactness of $\overline{\Omega}$, a positive maximum of the upper semicontinuous function $w_n$ (since $u$ is lower semicontinuous by assumption) must be achieved at some $x_n\in\Omega$, that is, there exists $x_n\in\Omega$ such that
\begin{equation}\label{eq:11'}
w_n(x_n)=\max_{x\in\Omega}w_n(x)>0.
\end{equation}
This implies together with \eqref{eq:5} that $0<u(x_n)<v_n(x_n)$. From this and thanks to \eqref{eq:11}, we find that
\begin{equation}\label{eq:12}
\lim\limits_{n\rightarrow\infty}u(x_n)=0.
\end{equation}
Recalling \eqref{eq:6}, we deduce from \eqref{eq:12} that $x_n\rightarrow\partial\Omega$ as $n\rightarrow\infty$. Taking this into account, one deduces that for any compact set $K\subset\Omega$ there exists $h>0$ such that $|x_n-y|\geq h>0$ for any $y\in K$ and $n$ sufficiently large. As a direct consequence, there exist two positive constants $\gamma_1,\gamma_2>0$, independent of $n$ such that
\begin{equation}\label{eq:13}
\gamma_1<\int_{K}\frac{dy}{{|x_n-y|}^{N+2s}}<\gamma_2\qquad\text{for $n$ sufficiently large (depending on $K$)}.
\end{equation}	
Thus we have 
\begin{equation}\label{eq:14}
c(x_n)u(x_n)\leq (-\Delta)^s_{\Omega}u(x_n)=c_{N,s}\int_{K}\frac{u(x_n)-u(y)}{{|x_n-y|}^{N+2s}}\;dy+c_{N,s}P.V.\int_{\Omega\setminus K}\frac{u(x_n)-u(y)}{{|x_n-y|}^{N+2s}}\;dy.
\end{equation}
We now aim at estimating the integrals on the right-hand side of the above inequality. Concerning the first integral, we notice that by \eqref{eq:6}, there exists a positive constant $\gamma_3>0$ such that $u(y)\geq\gamma_3$ for $y\in K$. As a consequence of this and by using \eqref{eq:12} and \eqref{eq:13}, it follows that
\begin{equation}\label{eq:15}
\limsup_{n\rightarrow\infty}\int_{K}\frac{u(x_n)-u(y)}{{|x_n-y|}^{N+2s}}\;dy\leq-\gamma_1\gamma_3<0.
\end{equation}
Regarding the second integral in \eqref{eq:14}, we first recall that since $x_n$ is the maximum of $w_n$ in $\overline{\Omega}$, then by \eqref{eq:11'} 
\begin{equation*}
u(x_n)-u(y)\leq v_n(x_n)-v_n(y).
\end{equation*}
Using this, the second integral in \eqref{eq:14} can be estimated as follows:
\begin{equation}\label{eq:16}
P.V.\int_{\Omega\setminus K}\frac{u(x_n)-u(y)}{{|x_n-y|}^{N+2s}}\ dy\leq P.V.\int_{\Omega\setminus K}\frac{v_n(x_n)-v_n(y)}{{|x_n-y|}^{N+2s}}\;dy.
\end{equation}
Moreover, a simple calculation yields
\begin{equation}\label{eq:17}
c_{N,s}P.V.\int_{\Omega\setminus K}\frac{v_n(x_n)-v_n(y)}{{|x_n-y|}^{N+2s}}\;dy=(-\Delta)^s_{\Omega}v_n(x_n)-c_{N,s}\int_{K}\frac{v_n(x_n)-v_n(y)}{{|x_n-y|}^{N+2s}}\;dy.
\end{equation}
Now, from \eqref{barrier} and \eqref{torsion-problem-for-regional}, it follows that
\begin{equation*}
(-\Delta)^s_{\Omega}v_n(x_n)=\frac{1}{n}(-\Delta)^s_{\Omega}u_{tor}(x_n)=\frac{1}{n}.
\end{equation*}
This yields
\begin{equation}\label{eq:13'}
(-\Delta)^s_{\Omega}v_n(x_n)\rightarrow0\quad\text{as}\quad n\rightarrow\infty.
\end{equation}
Combining \eqref{eq:13'}, \eqref{eq:11} and \eqref{eq:13}, we observe that the right-hand side in the equality \eqref{eq:17} goes to zero as $n\rightarrow\infty$ and therefore
\begin{equation*}
\lim_{n\rightarrow\infty}P.V.\int_{\Omega\setminus K}\frac{v_n(x_n)-v_n(y)}{{|x_n-y|}^{N+2s}}\;dy=0.
\end{equation*}
Consequently, from \eqref{eq:16}, we get
\begin{equation}\label{eq:19}
\limsup_{n\rightarrow\infty}P.V.\int_{\Omega\setminus K}\frac{u(x_n)-u(y)}{{|x_n-y|}^{N+2s}}\;dy\leq0.
\end{equation}
However, using that $c$ is bounded, it follows from \eqref{eq:12} that
\begin{equation}\label{eq:20}
\lim\limits_{n\rightarrow\infty}c(x_n)u(x_n)=0.
\end{equation} 
Finally, \eqref{eq:19} and \eqref{eq:15} into \eqref{eq:14}, lead to a contradiction with \eqref{eq:20}. Therefore, the inequality \eqref{eq:10} follows for some $n\in\N$ large enough.
\end{proof}

\section{Proof of the Hopf lemma: the case of weak super-solutions}\label{section:the case of weak solutions}
In this section, we aim at proving Theorem \ref{hopf-lemma-for-weak-super-solutions}. Here, the function $u_{tor}$ defined via \eqref{torsion-problem-for-regional} above is understood to be a weak solution. 
Recall the double-sided estimate \eqref{eq:8}. 
We first state and prove a technical lemma and a strong maximum principle for weak super-solutions of \eqref{main-problem-to-study}.

\begin{lemma}\label{lem:tech}
Let $\Omega\subset\R^N$ be an open bounded set and $c\in L^{\frac{N}{2s}}(\Omega)$. Then $u$ is a weak super-solution (in the sense of Definition \ref{definition-of-weak-supersolution}) of \eqref{main-problem-to-study} if and only if 
\begin{align*}
\cE(u,v)\geq\int_\Omega cuv \qquad\text{for any }v\in H^s_0(\Omega)_+.
\end{align*}
\end{lemma}

\begin{proof}
Fix $v\in H^s_0(\Omega)_+$ and let ${(\psi_n)}_{n\in\N}\subset C^\infty_c(\Omega)$ a sequence of nonnegative functions converging to $v$ in the $H^s(\Omega)$-norm. By Definition \ref{definition-of-weak-supersolution} we have
\begin{align*}
\cE(u,\psi_n)\geq\int_\Omega cu\psi_n
\qquad \text{for any }n\in\N.
\end{align*}
On the left-hand side we have the convergence $\cE(u,\psi_n)\to\cE(u,v)$ as $n\to\infty$ by construction; so, let us deal with right-hand side. By the Sobolev embedding we have $\psi_n\to v$ as $n\to\infty$ in $L^{2^*_s}(\Omega)$, with $2^*_s=\frac{2N}{N-2s}$. So, we have the convergence 
\begin{align*}
\int_\Omega cu\psi_n\longrightarrow\int_\Omega cuv\qquad\text{as }n\to\infty,
\end{align*}
if $cu\in L^{\frac{2N}{N+2s}}(\Omega)$ where $\frac{2N}{N+2s}=(2^*_s)'$ is the conjugate exponent of $2^*_s$, which is what we show next. This indeed follows from the H\"older inequality:
\begin{align*}
\int_\Omega\big|cu\big|^{\frac{2N}{N+2s}}=
\bigg(\int_\Omega\big|c\big|^{\frac{N}{2s}}\bigg)^{\frac{4s}{N+2s}}
\bigg(\int_\Omega\big|u\big|^{\frac{2N}{N-2s}}\bigg)^{\frac{N-2s}{N+2s}}<\infty.
\end{align*}
Then
\begin{align*}
\cE(u,v)=\lim_{n\to\infty}\cE(u,\psi_n)\geq\lim_{n\to\infty}\int_\Omega cu\psi_n=\int_\Omega cuv.
\end{align*}
\end{proof}

\begin{prop}[Strong maximum principle for weak super-solutions]\label{strong maximum principle}
	Let $c\in L^q(\Omega)$, with $q>\frac{N}{2s}$, and $u\in H^s(\Omega)\cap L^\infty_{loc}(\Omega)$  be a weak super-solution of
	\begin{equation}\label{e}
	(-\Delta)^s_{\Omega}u=c(x)u\qquad\text{in }\Omega.
	\end{equation}
	\begin{enumerate}[label=(\roman*)]
		\item If $c\leq0$ in $\Omega$ and $u\geq0$ on $\partial\Omega$, then either $u$ vanishes identically in $\Omega$ or $u>0$ in $\Omega$.
		\item If $u\geq0$ in $\Omega$, then either $u$ vanishes identically in $\Omega$ or $u>0$ in $\Omega$.
	\end{enumerate}
\end{prop}
\begin{proof}
	We first recall the following elementary inequality:
	\begin{equation}\label{good-inequality}
	\big(u(x)-u(y)\big)\big(u^{-}(x)-u^{-}(y)\big)\leq-\big(u^{-}(x)-u^{-}(y)\big)^2,
	\qquad\text{for any }x,y\in\Omega.
	\end{equation}
	Assume then $c\leq0$ in $\Omega$ and $u\geq0$ on $\partial\Omega$. Then $u^-=0$ on $\partial\Omega$. Moreover, by standard arguments, we also know $u^-\in H^s(\Omega)$. Therefore $u^{-}\in H^s_0(\Omega)_+$. Hence, by testing \eqref{e} on $u^-$ (which is allowed by Lemma \ref{lem:tech}), we have from inequality \eqref{good-inequality} that
	\begin{align*}
	\int_{\Omega}c(x)u(x)u^{-}(x)\;dx\leq\cE(u,u^-)\leq-\cE(u^-,u^-).
	\end{align*}
	Moreover, $u=u^{+}-u^-$ with $u^{+}u^{-}\equiv 0$ in $\Omega$, which yields
	\begin{align*}
	\int_{\Omega}c(x)u^{-}(x)^2\;dx\geq\cE(u^-,u^-)\geq\lambda_1(\Omega)\|u^-\|^2_{L^2(\Omega)},
	\end{align*}
	where $\lambda_1(\Omega)$ has been defined in \eqref{first-dirichlet-eigenvalue-of-regional}. Since $\lambda_1(\Omega)>0$, then from the nonpositivity of $c$ it  follows 
	\begin{equation*}
	\|u^-\|^2_{L^2(\Omega)}=0
	\end{equation*}
	implying that $u^-=0$ a.e. in $\Omega$, that is, $u\geq0$ a.e. in $\Omega$. 
	
	So we can at this point assume that $u\geq 0$ in $\Omega$. Note that the fact that $u$ is a weak super-solution implies in particular that $u$ is also a distributional super-solution. Indeed, for any $\psi\in\ C^\infty_c(\Omega)$, $\psi\geq 0$ in $\Omega$, 
	\begin{align*}
	\int_\Omega cu\psi\leq\cE(u,\psi)&=\frac{c_{N,s}}{2}\int_\Omega\int_\Omega\frac{\big(u(x)-u(y)\big)\big(\psi(x)-\psi(y)\big)}{{|x-y|}^{N+2s}}\;dx;dy \\
	&=c_{N,s}\int_\Omega u(x)\,P.V.\int_\Omega\frac{\psi(x)-\psi(y)}{{|x-y|}^{N+2s}}\;dy\;dx=\int_\Omega u(-\Delta)^s_\Omega\psi.
	\end{align*}
	Using this remark, we can use Proposition \ref{strong-max-principle-distribution-regional}.
\end{proof}

\begin{remark}\label{addendum strong maximum principle}
It is possible to drop the assumption $u\in L^\infty_{loc}(\Omega)$ in Proposition \ref{strong maximum principle} by paying the price of assuming $c\in L^\infty(\Omega)$. In this case, the first part of the proof still holds, while, instead of using Proposition \ref{strong-max-principle-distribution-regional}, the second part simply follows from \cite[Theorem 1.2]{jarohs2019strong}.
\end{remark}

We now prove Theorem \ref{hopf-lemma-for-weak-super-solutions}. For the sake of clarity, we split its proof into two different arguments.
\begin{proof}[Proof of Theorem \ref{hopf-lemma-for-weak-super-solutions} under assumption \eqref{ass1}]
Suppose that $u$ does not vanish identically in $\Omega$ and let us prove \eqref{c}. In other words, we want to prove that there exists a positive constant $C>0$ such that
\begin{equation}\label{lower-bound}
u\geq C\delta_{\Omega}^{2s-1}\qquad\text{in }\Omega.
\end{equation}
From Proposition \ref{strong maximum principle} and Remark \ref{addendum strong maximum principle} it follows that $u>0$ in $\Omega$. This means that for any $K\subset\subset\Omega$ there exists $\epsilon>0$ such that it holds
\begin{equation}\label{c0}
u(x)\geq\epsilon>0\qquad\text{for }x\in K.
\end{equation}
Now, let $w_n:=v_n-u$ where $v_n$ is the function defined in \eqref{barrier}. Then, thanks to \eqref{eq:11} and \eqref{c0}, we can assume without any ambiguity that
\begin{equation}\label{c1}
w_n^+\equiv0 \qquad \text{in }K\quad\text{for $n$ sufficiently large}.
\end{equation}
Now, since $w_n^+\in H^s_0(\Omega)_+$ (because $w_n^+\geq0$ in $\Omega$, $w_n^+\in H^s(\Omega)$ since $w_n$ is, and $w_n^+=0$ on $\partial\Omega$ since $v_n=0$ on $\partial\Omega$ and $u\geq0$ on $\partial\Omega$), one can use it as a test function in Definition \ref{definition-of-weak-supersolution} (by Lemma \ref{lem:tech}) in order to have
\begin{align}\label{a}
\cE(u,w_n^+)\geq\int_{\Omega}cuw_n^+. 
\end{align}
Since in $\{w_n^+>0\}$ it holds 
\[
u<v_n\leq\frac1n\|u_{tor}\|_{L^\infty(\Omega)},
\]
we have
\begin{align}\label{c2}
\int_{\Omega}c(x)u(x)w_n^+(x)\;dx\geq -\frac1n\|c\|_{L^\infty(\Omega)}\|u_{tor}\|_{L^\infty(\Omega)}\|w_n^+\|_{L^1(\Omega)}.
\end{align}
On the other hand,
\begin{multline*}
\cE(u,w^+_n)=\cE(u-v_n,w^+_n)+\cE(v_n,w^+_n)=\cE(-w_n,w^+_n)+\frac{1}{n}\cE(u_{tor},w^+_n)=\\
=-\cE(w_n^+,w^+_n)+\cE(w^-_n,w^+_n)+\frac{1}{n}\cE(u_{tor},w^+_n).
\end{multline*}
Since the first term on the right-hand side of the above equality is nonpositive and $\cE(u_{tor},w^+_n)=\int_{\Omega}w^+_n=\|w^+_n\|_{L^1(\Omega)}$ thanks to \eqref{torsion-problem-for-regional}, then
\begin{align}\label{c4}
\cE(u,w^+_n)\leq\cE(w^-_n,w^+_n)+\frac{1}{n}\|w^+_n\|_{L^1(\Omega)}.
\end{align}
Now, 
\begin{align}\label{c5}
\cE(w^-_n,w^+_n)
=-c_{N,s}\int_\Omega\int_\Omega\frac{w_n^-(x)w_n^+(y)}{{|x-y|}^{N+2s}}\;dx\;dy.
\end{align}
Recall that, by definition (see also \eqref{c1}), $K\subset\{w_n<0\}=\{w_n^->0\}$ and 
\[
w_n^-\geq\epsilon-\frac1n\|u_{tor}\|_{L^\infty(\Omega)}
\qquad\text{in }\Omega,
\]
so, upon plugging this into \eqref{c5}, we obtain 
\begin{align*}
\cE(w^-_n,w^+_n) & \leq -c_{N,s}\int_\Omega\int_K\frac{w_n^-(x)w_n^+(y)}{{|x-y|}^{N+2s}}\;dx\;dy \\
& \leq c_{N,s}\bigg(\frac1n\|u_{tor}\|_{L^\infty(\Omega)}-\epsilon\bigg)\int_\Omega\int_K\frac{w_n^+(y)}{{|x-y|}^{N+2s}}\;dx\;dy \\
& \leq C_0c_{N,s}\bigg(\frac1n\|u_{tor}\|_{L^\infty(\Omega)}-\epsilon\bigg)\|w_n^+\|_{L^1(\Omega)}\end{align*}
for some $C_0>0$ and $n$ sufficiently large. Plugging \eqref{c2} into \eqref{a}, using \eqref{c4} and this last obtained inequality, we get
\begin{multline}\label{c8}
C_0c_{N,s}\bigg(\frac1n\|u_{tor}\|_{L^\infty(\Omega)}-\epsilon\bigg)\|w_n^+\|_{L^1(\Omega)}
+\frac{1}{n}\|w^+_n\|_{L^1(\Omega)} \geq \\
\geq -\frac1n\|c\|_{L^\infty(\Omega)}\|u_{tor}\|_{L^\infty(\Omega)}\|w_n^+\|_{L^1(\Omega)}.
\end{multline}
For $n$ sufficiently large, we deduce from \eqref{c8} that
\begin{equation*}
w^+_n\equiv0\qquad\text{in }\Omega.
\end{equation*}
Therefore, \eqref{lower-bound} follows. 
\end{proof}

\begin{proof}[Proof of Theorem \ref{hopf-lemma-for-weak-super-solutions} under assumption \eqref{ass2}]
The very first part of the proof follows the argument given above. We start here from \eqref{a}. We know from Proposition \ref{strong maximum principle} that $u>0$ in $\Omega$ and so $w_n<v_n$, from which it follows 
\begin{align*}
\int_\Omega cuw_n^+\geq-\frac1n\|u_{tor}\|_{L^\infty(\Omega)}\int_\Omega|cu|.
\end{align*}
By the fractional Sobolev inequality we have that $u\in L^p(\Omega)$ for any $1\leq p\leq 2^*_s=2N/(N-2s)$. As the conjugate exponent of $2N/(N-2s)$ is $2N/(N+2s)$ which is smaller than $N/(2s)$, we have by an application of the H\"older's inequality that 
\begin{align*}
\int_\Omega cuw_n^+\geq-\frac1n\|u_{tor}\|_{L^\infty(\Omega)}\|u\|_{L^{2^*_s}(\Omega)}\|c\|_{L^q(\Omega)}.
\end{align*}
By repeating the calculations in the preceding argument we then get the analog of \eqref{c8} which reads in this case
\begin{multline*}
C_0c_{N,s}\bigg(\frac1n\|u_{tor}\|_{L^\infty(\Omega)}-\epsilon\bigg)\|w_n^+\|_{L^1(\Omega)}
+\frac{1}{n}\|w^+_n\|_{L^1(\Omega)} \geq \\
\geq -\frac1n\|u_{tor}\|_{L^\infty(\Omega)}\|u\|_{L^{2^*_s}(\Omega)}\|c\|_{L^q(\Omega)}.
\end{multline*}
This last inequality, for $n$ sufficiently large, gives
\begin{equation*}
w^+_n\equiv0\qquad\text{in }\Omega.
\end{equation*}
Therefore, \eqref{lower-bound} follows also in this case. 
\end{proof}

\section{Proof of the strong maximum principle for distributional super-solutions}\label{section:the case of distributional solutions}
This last section is devoted to the proof of Proposition \ref{strong-max-principle-distribution-regional}. In the following, we assume that $u:\Omega\to\R$ is a distributional super-solution (in the sense of Definition \ref{dfinition-of-distributional-supersolution}) of \eqref{main-problem-to-study} and that $c$ satisfies the assumptions in \eqref{q-hyp}.

%

\subsection{Regional v. restricted fractional Laplacian}
Note that 
\begin{align*}
(-\Delta)^s_\Omega\psi=(-\Delta)^s\psi-\kappa_\Omega\psi\quad\text{in }\Omega,
\qquad\kappa_\Omega(x)=c_{N,s}\int_{\R^N\setminus\Omega}\frac{dy}{{|x-y|}^{N+2s}} \quad \text{for }x\in\Omega,
\end{align*}
where we recall \eqref{def-fractional},
so that Definition \ref{dfinition-of-distributional-supersolution} is equivalent to (if we extend $u=0$ in $\R^N\setminus\overline{\Omega}$)
\begin{align}\label{eq2}
(-\Delta)^s u\geq \big(c+\kappa_\Omega\big)u \quad\text{in }\cD'(\Omega),
\qquad c\in L^q_{\text{loc}}(\Omega),\ q>\frac{N}{2s}.
\end{align}

\subsection{Approximation and representation of distributional solutions}

Consider a solution $u:\R^N\to[0,+\infty)$ to \eqref{eq2} with
\begin{align}\label{z0}
u\in L^\alpha_{\text{loc}}(\Omega),\qquad \alpha>\frac{Nq}{2sq-N},~
\text{and}\quad u=0 \text{ in }\R^N\setminus\overline{\Omega}.
\end{align}
Take $\eta_\eps\in C^\infty_c(B_\eps)$ a mollifier. If we take an open $\Omega'\subset\subset\Omega$ then for any $\psi\in C^\infty_c(\Omega'),\ \psi\geq 0,$ it holds $\psi*\eta_\eps\in C^\infty_c(\Omega)$ for $\eps$ small independently of $\psi$ and we can say
\begin{multline*}
\int_{\R^N}\big(u*\eta_\eps\big)(-\Delta)^s\psi=\int_{\R^N}u\big(\eta_\eps*(-\Delta)^s\psi\big)=\int_\Omega u(-\Delta)^s\big(\psi*\eta_\eps)\geq\int_\Omega\big(c+\kappa_\Omega\big)u\big(\psi*\eta_\eps\big) 
=\\= 
\int_\Omega\Big[\Big(\big(c+\kappa_\Omega\big)u\Big)*\eta_\eps\Big]\psi
\end{multline*}
which implies that
\begin{align*} 
(-\Delta)^s\big(u*\eta_\eps\big)\geq\Big(\big(c+\kappa_\Omega\big)u\Big)*\eta_\eps 
\qquad \text{in }\Omega'.
\end{align*}
As $u*\eta_\eps\in C^\infty(\overline{\Omega'})$, the above inequality also holds in a pointwise sense.
We can then exploit a Green representation on $u*\eta_\eps$ (see \cite{bucur2016some}) to deduce that for any $x\in\Omega''\subset\subset\Omega'$ and $0<r<\dist(\Omega'',\R^N\setminus\Omega')$
\begin{multline}\label{p}
\big(u*\eta_\eps\big)(x)\geq r^{2s}\int_{B_1}G(0,y)\Big[\Big(\big(c+\kappa_\Omega\big)u\Big)*\eta_\eps\Big](x+ry)\;dy\;+\\
+\int_{\R^N\setminus B_1}P(0,y)\big(u*\eta_\eps\big)(x+ry)\;dy.
\end{multline}
Here we have used the kernels $G$ and $P$ which are respectively the Green function and the Poisson  kernel of the fractional Laplacian $(-\Delta)^s$ on the unitary ball $B_1$, which are explicitly known, see \cite{bucur2016some}:
\begin{align*}
G(x,y)&=\frac{k_{N,s}}{{|x-y|}^{N-2s}}\int_0^{\frac{(1-|x|^2)(1-|y|^2)}{|x-y|^2}}\frac{t^{s-1}}{(t+1)^{N/2}}\;dt && x,y\in B_1, \\
P(x,y)&=\frac{\gamma_{N,s}}{|x-y|^N}\bigg(\frac{1-|x|^2}{|y|^2-1}\bigg)^s && x\in B_1,\ y\in\R^N\setminus B_1.
\end{align*}

From now on, we assume that $u\geq0$ in $\Omega$. We want to send $\eps\to 0$ in \eqref{p} and deduce a representation for $u$.
For the Poisson integral we use the nonnegativity of $u$ and the Fatou's Lemma to say
\begin{align*}
\liminf_{\eps\to 0}\int_{\R^N\setminus B_1}P(0,y)\big(u*\eta_\eps\big)(x+ry)\;dy
\geq \int_{\R^N\setminus B_1}P(0,y)\,u(x+ry)\;dy.
\end{align*}
For the Green integral we use that
\begin{align*}
G(0,\cdot)\in L^p(B_1)\qquad\text{for any }p\in\Big[1,\frac{N}{N-2s}\Big)
\end{align*}
and
\begin{multline*}
\Big\|\Big(\big(c+\kappa_\Omega\big)u\Big)*\eta_\eps\Big\|_{L^\beta(\Omega')}\leq
C\Big\|\big(c+\kappa_\Omega\big)u\Big\|_{L^\beta(\Omega')}\leq 
C\|cu\|_{L^\beta(\Omega')}+C\big\|\kappa_\Omega\big\|_{L^\infty(\Omega')} \|u\|_{L^\alpha(\Omega')} \\
\text{for any }\beta\in\Big(\frac{N}{2s},\alpha\Big)
\end{multline*}
where, moreover, by the H\"{o}lder inequality
\begin{align*}
\int_{\Omega'}{|cu|}^\beta \leq 
\|c\|^{\frac1\beta}_{L^q(\Omega')}\big\|u^\beta\big\|_{L^{q/(q-\beta)}(\Omega')} &\qquad
\text{for }\ \frac{N}{2s}<\beta<q, \\
\int_{\Omega'}u^{\beta q/(q-\beta)}<\infty & \qquad 
\text{for }\ \frac{\beta q}{q-\beta}<\alpha,
\end{align*}
where the second inequality holds for $\beta$ close to $\frac{N}{2s}$ in view of \eqref{z0}.
Therefore, using the weak topology in Lebesgue spaces,
\begin{multline*}
\lim_{\eps\to 0}\int_{B_1}G(0,y)\Big[\Big(\big(c+\kappa_\Omega\big)u\Big)*\eta_\eps\Big](x+ry)\;dy=\\=\int_{B_1}G(0,y)\,\big(c+\kappa_\Omega\big)(x+ry)\,u(x+ry)\;dy.
\end{multline*}
Thus
\begin{multline}\label{repr}
u(x)\geq r^{2s}\int_{B_1}G(0,y)\,\big(c+\kappa_\Omega\big)(x+ry)\,u(x+ry)\;dy+
\int_{\R^N\setminus B_1}P(0,y)\,u(x+ry)\;dy \\
\text{for a.e. }x\in\Omega''.
\end{multline}

\subsection{The Hardy-Littlewood maximal function}
Recall that, given $f\in L^p_{\text{loc}}(\R^N)$, $p>1$, the Hardy-Littlewood maximal function is defined as
\begin{align}\label{hl}
\mathbf{M}[f](x)=\sup_{r>0}\frac1{r^N}\int_{B_r(x)}|f|
\qquad x\in\R^N.
\end{align}
In the following we are going to need the following fact
\begin{align}\label{hl-ineq}
\big\|\mathbf{M}[f]\big\|_{L^p(K)}\leq C\|f\|_{L^p(K)}
\qquad\text{for }p>1\text{ and }K\subset\subset\R^N\text{ measurable}.
\end{align}

\subsection{The strong maximum principle}
Having the above ingredients, in this subsection, we are ready to give the proof of Proposition \ref{strong-max-principle-distribution-regional}.

\begin{proof}[Proof of Proposition \ref{strong-max-principle-distribution-regional}]
We argue by contradiction.
	Assume that $|\{u>\delta\}\cap\Omega'|>0$ for some $\delta>0$.
	
	In the notations of the previous subsection, and without loss of generality, we assume that
	\begin{align*}
	&\text{there exist}\ \ {(x_j)}_{j\in\N}\subset\Omega'' \ \text{ and }\ \ {(r_j)}_{j\in\N}\subset(0,\infty),\ r_j\to 0\text{ as } j\to\infty,\\ 
	&\text{ such that }\ \lim_{j\to\infty}\frac1{(2r_j)^N}\int_{B_{2r_j}(x_j)}u=0.
	\end{align*}
	Without loss of generality, we can assume that ${(r_j)}_{j\in\N}$ is decreasing.
	Extract a subsequence ${(\rho_j)}_{j\in\N}\subset{(r_j)}_{j\in\N}$ in such a way that\footnote{Here we briefly comment on inequality \eqref{rhoj}. As we know by assumption that $\frac{1}{(2r_j)^N}\int_{B_{2r_j}(x_j)}u\rightarrow0$ as $j\to\infty$, one has also $\frac{2^N}{(2r_j)^N}\int_{B_{2r_j}(x_j)}u\rightarrow0$ as $j\to\infty$. Now, using that $B_{r_j}(x_j)\subset B_{2r_j}(x_j)$ and that $u$ is nonnegative, one can write
\[
0\leq\frac1{r_j^N}\int_{B_{r_j}(x_j)}u\leq\frac{2^N}{(2r_j)^N}\int_{B_{2r_j}(x_j)}u\longrightarrow 0\qquad\text{as }j\to\infty. 
\]
One can then extract a subsequence $(\rho_j)_{j\in\N}\subset(r_j)_{j\in\N}$ with $\rho_j\leq r_j$ such that \eqref{rhoj} holds.
}
	\begin{align}\label{rhoj}
	\frac1{\rho_j^N}\int_{B_{\rho_j}(x_j)}u\leq\frac{r_j^{2s}}{j}
	\qquad\text{and}\qquad 
	\rho_j\leq r_j
	\qquad\text{for any }j\in\N.
	\end{align}
	In order to ease notation, relabel $c_\Omega=c+\kappa_\Omega$.
	We apply representation \eqref{repr} with $r=r_j$ and we then integrate it over $B_{\rho_j}(x_j)$,
	obtaining
	\begin{multline}\label{to be estimated}
	\frac1{\rho_j^N}\int_{B_{\rho_j}(x_j)}u\geq \frac{r_j^{2s}}{\rho_j^N}\int_{B_1}G(0,y)\int_{B_{\rho_j}(x_j)}c_\Omega(x+r_jy)\,u(x+r_jy)\;dx\;dy\;+\\
	+\frac1{\rho_j^N}\int_{\R^N\setminus B_1}P(0,y)\int_{B_{\rho_j}(x_j)}u(x+r_jy)\;dx\;dy.
	\end{multline}
	
	The Poisson integral can be estimated as follows:
	\begin{align*}
	\int_{\R^N\setminus B_1}P(0,y)\,u(x+r_jy)\;dy &=
	\gamma_{N,s}\int_{\R^N\setminus B_1}\frac{u(x+r_jy)}{{|y|}^N\big(|y|^2-1\big)^s}\;dy \\
	&\geq
	\gamma_{N,s}r_j^{2s}\int_{\Omega'\setminus B_{r_j}(x)}\frac{u(y)}{{|y-x|}^N\big(|y-x|^2-r_j^2\big)^s}\;dy\\
	&\geq 
	Cr_j^{2s}\int_{\Omega'\setminus B_{r_j}(x)}u
	\end{align*}
	which entails
	\begin{align}\label{poisson-est}
	\frac1{\rho_j^N}\int_{\R^N\setminus B_1}P(0,y)\int_{B_{\rho_j}(x_j)}u(x+r_jy)\;dx\;dy
	\geq Cr_j^{2s}
	\end{align}
	for some $C>0$. Mind that here we have used the assumption that $|\{u>\delta\}\cap\Omega'|>0$ for some $\delta>0$.
	
	We now deal with the Green integral in \eqref{to be estimated}.
	Fix $p\in(1,\min\{q,N/(N-2s)\})$. We estimate
	\begin{align}
	& \frac1{\rho_j^N}\int_{B_1}G(0,y)\int_{B_{\rho_j}(x_j)}c_\Omega(x+r_jy)\,u(x+r_jy)\;dx\;dy\geq \nonumber \\
	& \geq -\frac{C}{\rho_j^N}\int_{B_1}|y|^{2s-N}\int_{B_{\rho_j}(x_j)}|c_\Omega(x+r_jy)|\,u(x+r_jy)\;dx\;dy \nonumber \\
	& \geq -C\int_{B_1}|y|^{2s-N}\bigg(\frac1{\rho_j^N}\int_{B_{\rho_j}(x_j)}|c_\Omega(x+r_jy)|^{\overline{q}}\;dx\bigg)^\frac1{\overline{q}}\times \nonumber \\
	& \qquad\times\bigg(\frac1{\rho_j^N}\int_{B_{\rho_j}(x_j)}u(x+r_jy)^{\frac{\overline{q}}{\overline{q}-1}}\;dx\bigg)^\frac{\overline{q}-1}{\overline{q}}\;dy \nonumber \\
	&\geq
	-C\bigg[\int_{B_1}|y|^{(2s-N)p}\bigg(\frac1{\rho_j^N}\int_{B_{\rho_j}(x_j)}|c_\Omega(x+r_jy)|^{\overline{q}}\;dx\bigg)^\frac{p}{\overline{q}}\;dy\bigg]^{\frac1p}\;\times \label{1111111}\\
	&\qquad\times\;\bigg[\int_{B_1}\bigg(\frac1{\rho_j^N}\int_{B_{\rho_j}(x_j)}|u(x+r_jy)|^{\frac{\overline{q}}{\overline{q}-1}}\;dx\bigg)^{\frac{p}{p-1}\frac{\overline{q}-1}{\overline{q}}}\;dy\bigg]^{\frac{p-1}p}. \label{2222222}
	\end{align}
	Using that
	\begin{align*}
	\frac1{\rho_j^N}\int_{B_{\rho_j}(x_j)}|c_\Omega(x+r_jy)|^{\overline{q}}\;dx\leq\mathbf{M}\big[|c_\Omega|^{\overline{q}}\big](x_j+r_jy)
	\end{align*}
	by definition \eqref{hl}, we obtain for \eqref{1111111} the following estimates by means of a H\"{o}lder inequality 
	\begin{align}
	& \int_{B_1}|y|^{(2s-N)p}\bigg(\frac1{\rho_j^N}\int_{B_{\rho_j}(x_j)}|c_\Omega(x+r_jy)|^{\overline{q}}\;dx\bigg)^\frac{p}{\overline{q}}\;dy\leq \nonumber \\
	& \leq\int_{B_1}|y|^{(2s-N)p}\mathbf{M}\big[|c_\Omega|^{\overline{q}}\big](x_j+r_jy)^\frac{p}{\overline{q}}\;dy \nonumber \\
	& \leq\bigg(\int_{B_1}|y|^{\frac{(2s-N)pq}{q-p}}\;dy\bigg)^{\frac{q-p}{pq}}\bigg(\int_{B_1}\mathbf{M}\big[|c_\Omega|^{\overline{q}}\big](x_j+r_jy)^\frac{q}{\overline{q}}\;dy\bigg)^{\frac{p}{q}} \nonumber \\
	& \leq \bigg(\int_{B_1}|y|^{\frac{(2s-N)pq}{q-p}}\;dy\bigg)^{\frac{q-p}{pq}} \big\|\mathbf{M}\big[|c_\Omega|^{\overline{q}}\big]\big\|^{p/\overline{q}}_{L^{q/\overline{q}}(\Omega')} \nonumber \\
	& \leq \bigg(\int_{B_1}|y|^{\frac{(2s-N)pq}{q-p}}\;dy\bigg)^{\frac{q-p}{pq}} \big\||c_\Omega|^{\overline{q}}\big\|^{p/\overline{q}}_{L^{q/\overline{q}}(\Omega')} \nonumber \\
	& \leq \bigg(\int_{B_1}|y|^{\frac{(2s-N)pq}{q-p}}\;dy\bigg)^{\frac{q-p}{pq}} \big\|c_\Omega\big\|^{p}_{L^q(\Omega')}  \label{cbnwjkb}
	\end{align}
	by \eqref{hl-ineq}. Remark that the assumption $1<p<N/(N-2s)$ ensures that
	\begin{align*}
	\frac{(2s-N)pq}{q-p}>-\frac{Nq}{q-p}>-N,
	\end{align*}
	which guarantees the finiteness of the first factor in \eqref{cbnwjkb}.
	
	Fix now $\overline{q}\in(p,q)$ and notice how this implies
	\begin{align*}
	\frac{p}{p-1}\frac{\overline{q}-1}{\overline{q}}>1.
	\end{align*}
	Using this, we estimate \eqref{2222222} as follows:
	\begin{align*}
	& \int_{B_1}\bigg(\frac1{\rho_j^N}\int_{B_{\rho_j}(x_j)}|u(x+r_jy)|^{\frac{\overline{q}}{\overline{q}-1}}\;dx\bigg)^{\frac{p}{p-1}\frac{\overline{q}-1}{\overline{q}}}\;dy
	\leq \\
	&\leq C
	\big\|u\big\|_{L^\infty(\Omega')}^{\frac{\overline{q}}{\overline{q}-1}(\frac{p}{p-1}\frac{\overline{q}-1}{\overline{q}}-1)}
	\int_{B_1}\bigg(\frac1{\rho_j^N}\int_{B_{\rho_j}(x_j)}|u(x+r_jy)|^{\frac{\overline{q}}{\overline{q}-1}}\;dx\bigg)\;dy \\
	&\leq C
	\big\|u\big\|_{L^\infty(\Omega')}^{(\frac{p}{p-1}-\frac{\overline{q}}{\overline{q}-1})+(\frac{\overline{q}}{\overline{q}-1}-1)}
	\int_{B_1}\bigg(\frac1{\rho_j^N}\int_{B_{\rho_j}(x_j)}u(x+r_jy)\;dx\bigg)\;dy \\
	&= C
	\big\|u\big\|_{L^\infty(\Omega')}^{\frac1{p-1}}
	\int_{B_1}\bigg(\frac1{\rho_j^N}\int_{B_{\rho_j}(x_j)}u(x+r_jy)\;dx\bigg)\;dy.
	\end{align*}
	Note that
	\begin{align}
	& \int_{B_1}\frac1{\rho_j^N}\int_{B_{\rho_j}(x_j)}u(x+r_jy)\;dx\;dy = 
	\frac1{r_j^N}\int_{B_{r_j}}\frac1{\rho_j^N}\int_{B_{\rho_j}}u(x_j+x+y)\;dx\;dy \nonumber \\
	& =\frac1{\rho_j^N}\int_{B_{\rho_j}}\frac1{r_j^N}\int_{B_{r_j}}u(x_j+x+y)\;dy\;dx 
	\leq\frac1{\rho_j^N}\int_{B_{\rho_j}}\frac1{r_j^N}\int_{B_{r_j+\rho_j}}u(x_j+z)\;dz\;dx \nonumber \\
	& =\omega_N\Big(\frac{r_j+\rho_j}{r_j}\Big)^N\frac1{\big(r_j+\rho_j\big)^N}\int_{B_{r_j+\rho_j}}u(x_j+z)\;dz \nonumber\\
	&
	\leq\frac{C}{\big(r_j+\rho_j\big)^N}\int_{B_{r_j+\rho_j}}u(x_j+z)\;dz\nonumber \\
	& \leq C\Big(\frac{2r_j}{r_j+\rho_j}\Big)^N\frac1{(2r_j)^N}\int_{B_{2r_j}}u(x_j+z)\;dz 
	\ \longrightarrow\;0 \qquad\text{as }j\to\infty.\label{green2-est}
	\end{align}
	
	We therefore deduce, by plugging in \eqref{to be estimated} the estimates contained in \eqref{rhoj}, \eqref{poisson-est}, \eqref{cbnwjkb}, and \eqref{green2-est},
	\begin{align*}
	\frac{r_j^{2s}}{j}\geq -C_1r_j^{2s}\eps_j+C_2r_j^{2s},
	\qquad\text{for some }{(\eps_j)}_{j\in\N}\subset(0,\infty),\ \eps_j\to 0\text{ as }j\to\infty.
	\end{align*}
	But this gives a contradiction for $j$ large enough.
\end{proof}

\bibliographystyle{ieeetr}

\end{document}